\documentclass{article}
\usepackage{latexsym,amsfonts,amsmath,amsthm,makeidx}
\usepackage{makeidx}

\makeindex
\newtheorem{definition}{Definition}[section]
\newtheorem{theorem}{Theorem}[section]
\newtheorem{lemma}{Lemma}[section]
\newtheorem{corollary}{Corollary}[section]
\newtheorem{proposition}{Proposition}[section]
\newtheorem{remark}{Remark}[section]
\newtheorem{example}{Example}[section]
\newcommand{\RN}{\mathbb R^N}

\newcommand{\iy}{\infty}

\newcommand{\s}{\section}
\newcommand{\dd}{\delta}
\newcommand{\DD}{\Delta}
\newcommand{\g}{\gamma}
\newcommand{\G}{\Gamma}
\newcommand{\na}{\nabla}

\newcommand{\la}{\lambda}

\newcommand{\pa}{\partial}
\newcommand{\si}{\sigma}

\newcommand{\R}{\mathbb R}

\newcommand{\ti}{\tilde}

\newcommand{\rg}{\rightarrow}

\newcommand{\e}{\varepsilon}
\newcommand{\vp}{\varphi}

\newcommand{\lab}{\label}
\newcommand{\bt}{\begin{theorem}}
\newcommand{\et}{\end{theorem}}
\newcommand{\bl}{\begin{lemma}}
\newcommand{\el}{\end{lemma}}
\newcommand{\bd}{\begin{definition}}
\newcommand{\ed}{\end{definition}}
\newcommand{\bc}{\begin{corollary}}
\newcommand{\ec}{\end{corollary}}
\newcommand{\bp}{\begin{proof}}
\newcommand{\ep}{\end{proof}}
\newcommand{\bx}{\begin{example}}
\newcommand{\ex}{\end{example}}
\newcommand{\bi}{\begin{exercise}}
\newcommand{\ei}{\end{exercise}}
\newcommand{\bo}{\begin{proposition}}
\newcommand{\eo}{\end{proposition}}
\newcommand{\br}{\begin{remark}}
\newcommand{\er}{\end{remark}}
\newcommand{\be}{\begin{equation}}
\newcommand{\ee}{\end{equation}}
\newcommand{\ba}{\begin{align}}
\newcommand{\ea}{\end{align}}
\newcommand{\bn}{\begin{enumerate}}
\newcommand{\en}{\end{enumerate}}
\newcommand{\bg}{\begin{align*}}
\newcommand{\bcs}{\begin{cases}}
\newcommand{\ecs}{\end{cases}}

\newcommand{\Sg}{\Sigma}

\newcommand{\bean}{\begin{eqnarray*}}
\newcommand{\eean}{\end{eqnarray*}}

\numberwithin{equation}{section}
\begin{document}

\title{\bf{Infinitely many sign-changing solutions for the nonlinear
Schr\"{o}dinger-Poisson system}}
\date{}
\author{{\bf Zhaoli Liu$^1$\thanks{Supported by NSFC (11271265, 11331010)
and BCMIIS. Email: zliu@cnu.edu.cn}, Zhi-Qiang Wang$^{
3}$\thanks{Supported by NSFC (11271201). Email:
zhi-qiang.wang@usu.edu}, and Jianjun Zhang$^{2}$\thanks{Supported by
CPSF (2013M530868). Email: zhangjianjun09@tsinghua.org.cn}}\\\\
\footnotesize {\it $^1$School of Mathematical Sciences, Capital
Normal University, Beijing 100037, PR China}\\
\footnotesize {\it $^2$Chern Institute of Mathematics, Nankai
University, Tianjin 300071, PR China}\\
\footnotesize {\it $^3$Department of Mathematics and Statistics,
Utah State University, Logan, Utah 84322, USA}}

\numberwithin{equation}{section}

\maketitle 

\begin{center}
\begin{minipage}{120mm}
\begin{center}{\bf Abstract}\end{center}

In this paper, we consider the following Schr\"{o}dinger-Poisson
system
\begin{eqnarray*} \left\{
\begin{array}{ll}
-\Delta u+V(x)u+\phi u=f(u)&\mbox{in}\ \R^3,\\
-\DD\phi=u^2&\mbox{in}\ \R^3.
\end{array}
\right.
\end{eqnarray*}
We investigate the existence of multiple bound state solutions, in
particular sign-changing solutions. By using the method of invariant
sets of descending flow, we prove that this system has infinitely
many sign-changing solutions. In particular, the nonlinear term
includes the power-type nonlinearity $f(u)=|u|^{p-2}u$ for the
well-studied case $p\in(4,6)$, and the less-studied case
$p\in(3,4)$, and for the latter case few existence results are
available in the literature.


\end{minipage}
\end{center}


\s{Introduction and main results}
\renewcommand{\theequation}{1.\arabic{equation}}

In this paper, we are concerned with the existence of bound state
solutions, in particular sign-changing solutions, to the following
nonlinear Schr\"{o}dinger-Poisson system
\begin{equation}\lab{q1} \left\{
\begin{array}{ll}
-\Delta u+V(x)u+\phi u=f(u)&\mbox{in}\ \R^3,\\
-\DD\phi=u^2&\mbox{in}\ \R^3.
\end{array}
\right.
\end{equation}
In the last two decades, system (\ref{q1}) has been studied
extensively due to its strong physical background. From a physical
point of view, it describes systems of identical charged particles
interacting each other in the case that magnetic effects could be
ignored and its solution is a standing wave for such a system. The
nonlinear term $f$ models the interaction between the particles
\cite{Vaira}. The first equation of (\ref{q1}) is coupled with a
Poisson equation, which means that the potential is determined by
the charge of the wave function. The term $\phi u$ is nonlocal and
concerns the interaction with the electric field. For more detailed
physical aspects of systems like (\ref{q1}) and for further
mathematical and physical interpretation, we refer to
\cite{Ruiz1,Benci1,Benci2} and the references therein.

In recent years, there has been increasing attention to systems like
(\ref{q1}) on the existence of positive solutions, ground states,
radial and non-radial solutions and semiclassical states. Ruiz
\cite{Ruiz} considered the following problem
\begin{equation}\lab{qqqq1} \left\{
\begin{array}{ll}
-\Delta u+u+\lambda\phi u=|u|^{p-2}u&\mbox{in}\ \R^3,\\
-\DD\phi=u^2&\mbox{in}\ \R^3
\end{array}
\right.
\end{equation}
and gave existence and nonexistence results, depending on the
parameters $p\in(2,6)$ and $\lambda>0$. In particular, if
$\la\ge\frac{1}{4}$, the author showed that $p=3$ is a critical
value for the existence of positive solutions. By using the
concentration compactness principle, Azzollini and Pomponio
\cite{Az1} proved the existence of a ground state solution of
(\ref{q1}) when $f(u)=|u|^{p-2}u$ and $p\in(3,6)$. But no symmetry
information concerning this ground state solution was given. In
\cite{Ruiz2}, Ruiz studied the profile of the radial ground state
solutions to (\ref{qqqq1}) as $\lambda\rg 0$ for
$p\in(\frac{18}{7},3)$. Using variational method together with a
perturbation argument, Ambrosetti \cite{Am} investigated the
multiplicity of solutions and semiclassical states to systems like
(\ref{q1}). Here, we would also like to mention the papers
\cite{Az,Wei,T,Ianni3,Peng,Wang1} for related topics.

Another topic which has increasingly received interest in recent
years is the existence of sign-changing solutions of systems like
(\ref{q1}). Recall that a solution $(u,\phi)$ to \eqref{q1} is
called a sign-changing solution if $u$ changes its sign. Using a
Nehari-type manifold and gluing solution pieces together, Kim and
Seok \cite{Kim} proved the existence of radial sign-changing
solutions with prescribed numbers of nodal domains for (\ref{q1}) in
the case where $V(x)=1$, $f(u)=|u|^{p-2}u$, and $p\in(4,6)$. Ianni
\cite{Ianni1} obtained a similar result to \cite{Kim} for
$p\in[4,6)$, via a heat flow approach together with a limit
procedure. Recently, with a Lyapunov-Schmidt reduction argument,
Ianni and Vaira \cite{Ianni2} constructed non-radial multi-peak
solutions with arbitrary large numbers of positive peaks and
arbitrary large numbers of negative peaks to the
Schr\"{o}dinger-Poisson system
\begin{equation}\lab{qqqq} \left\{
\begin{array}{ll}
-\e^2\Delta u+u+\phi u=f(u)&\mbox{in}\ \R^N,\\
-\DD\phi=a_Nu^2&\mbox{in}\ \R^N
\end{array}
\right.
\end{equation}
for $\epsilon>0$ small, where $3\le N\le6$ and $a_N$ is a positive
constant. All the sign-changing solutions obtained in \cite{Kim,
Ianni1, Ianni2} have certain types of symmetries; they are either
$O(N)$-invariant or $G$-invariant for some finite subgroup $G$ of
$O(N)$ and thus the system is required to have a certain group
invariance. Based on variational method and Brouwer degree theory,
Wang and Zhou \cite{Zhou} obtained a least energy sign-changing
solution to (\ref{q1}) without any symmetry by seeking minimizer of
the energy functional on the sign-changing Nehari manifold when
$f(u)=|u|^{p-2}u$ and $p\in(4,6)$. More recently, in the case where
the system is considered on bounded domains $\Omega\subset\R^3$,
Alves and Souto \cite{Alves} obtained a similar result to
\cite{Zhou} for a more general nonlinear term $f$.




To the best of our knowledge, there is no result in the literature
on the existence of multiple sign-changing solutions as bound states
to problem (\ref{q1}) without any symmetry, and thus to prove the
existence of infinitely many sign-changing solutions to problem
(\ref{q1}) without any symmetry is the first purpose of the present
paper. Since the approaches in \cite{Alves,Ianni1,Kim,Zhou}, when
applied to the monomial nonlinearity $f(u)=|u|^{p-2}u$, are only
valid for $p\geq4$, we want to provide an argument which covers the
case $p\in(3,4)$ and this is the second purpose of the present
paper. Moreover, our method does not depend on existence of the
Nehari manifold.


In what follows, we assume $V\in C(\R^3,\R^+)$ satisfies the
following condition.
\begin{itemize}
\item [($V_0$)] $V$ is coercive, i.e., $\lim\limits_{|x|\rg\iy}V(x)=\iy$.
\end{itemize}
Moreover, we assume $f$ satisfies the following hypotheses.
\begin{itemize}
\item [($f_1$)] $f\in C(\R,\R)$ and
   $\lim\limits_{s\rightarrow 0}\frac{f(s)}{s}=0$.
\item [($f_2$)] $\limsup\limits_{|s|\rightarrow
   +\infty}\frac{|f(s)|}{|s|^{p-1}}<\iy$ for some $p\in (3,6)$.
\item [($f_3$)] There exists $\mu>3$ such that
   $tf(t)\ge \mu F(t)>0$ for all $t\not=0$, where $F(t)=\int_0^tf(s)ds$.
\end{itemize}

As a consequence of $(f_2)$ and $(f_3)$, one has $3<\mu\leq p<6$.
Our first result reads as

\begin{theorem}\lab{Th1}
If $(V_0)$ and $(f_1)$-$(f_3)$ hold and $\mu>4$, then problem
\eqref{q1} has one sign-changing solution. If moreover $f$ is odd,
then problem \eqref{q1} has infinitely many sign-changing solutions.
\end{theorem}


\begin{remark}
{\rm Assumption $(V_0)$ is used only in deriving compactness (the
(PS) condition) of the energy functional associated to (\ref{q1}).
If $\R^3$ in problem (\ref{q1}) is replaced with a smooth bounded
domain $\Omega\subset\R^3$, Theorem \ref{Th1} without $(V_0)$ and
any symmetry assumption on $\Omega$ still holds.}
\end{remark}

$(f_3)$ is the so-called Ambrosetti-Rabinowitz condition ((AR) for
short). Since the nonlocal term $\int_{\R^3}\phi_uu^2$ in the
expression of $I$ (see Section 2) is homogeneous of degree 4, if
$\mu$ from $(f_3)$ satisfies $\mu>4$ then (AR) guarantees
boundedness of (PS)-sequences as well as existence of a mountain
pass geometry in the sense that $I(tu)\to -\infty$ as $ t\to \infty$
for each $u\neq 0$. If $\mu<4$, (PS)-sequences may not be bounded
and one has $I(tu)\to \infty$ as $ t\to \infty$ for each $u\neq 0$.
To overcome these difficulties in the case $\mu<4$ we impose on $V$
an additional condition
\begin{itemize}
\item [($V_1$)] $V$ is differentiable, $\na V(x)\cdot x\in
L^r(\R^3)$ for some $r\in[\frac{3}{2},\iy]$ and
$$
2V(x)+\na V(x)\cdot x\ge 0\ \mbox{for a.e.}\ x\in\R^3.
$$
\end{itemize}
This assumption was introduced in \cite{ZZ1,ZZ2} in order to prove
compactness with the monotonicity trick of Jeanjean \cite{Jean}.
That $\na V(x)\cdot x\in L^r(\R^3)$ for some $r\in[\frac{3}{2},\iy]$
plays a role only in deriving the Pohoz$\check{\rm a}$ev identity
for solutions of \eqref{m} in Section 4, and it can clearly be
weakened since solutions of \eqref{m} decay at infinity.
Nevertheless, we do not want to go further in that direction. We
state our second result as follows.


\begin{theorem}\lab{Th2}
If $(V_0)$-$(V_1)$ and $(f_1)$-$(f_3)$ hold, then problem \eqref{q1}
has one sign-changing solution. If in addition $f$ is odd, then
problem \eqref{q1} has infinitely many sign-changing solutions.
\end{theorem}

\begin{remark}
{\rm The class of nonlinearities $f$ satisfying the assumptions of
Theorem \ref{Th2} includes the monomial nonlinearity
$f(u)=|u|^{p-2}u$ with $p\in(3,4)$. Even in this special case,
Theorem \ref{Th2} seems to be the first attempt in finding
sign-changing solutions to (\ref{q1}).}
\end{remark}

The idea of the proofs of Theorems \ref{Th1} and \ref{Th2} is to use
suitable minimax arguments in the presence of invariant sets of a
descending flow for the variational formulation. In particular we
make use of an abstract critical point theory developed by J. Liu,
X. Liu and Z.-Q. Wang \cite{Wang}. The method of invariant sets of
descending flow plays an important role in the study of
sign-changing solutions of elliptic problems; we refer to
\cite{Liu3,Liu2,Liu4,BPW,BW1,BW2,Liu1,LW} and the references
therein. However, with the presence of the coupling term $\phi u$,
the techniques of constructing invariant sets of descending flow in
\cite{Liu3,Liu2,Liu4,BPW,BW1,BW2,Liu1,LW} can not be directly
applied to system \eqref{q1}, which makes the problem more
complicated. The reason is that $\phi u$ is a non-local term and the
decomposition
$$
\int_{\R^3}\phi_u|u|^2=\int_{\R^3}\phi_{u^+}|u^+|^2+\int_{\R^3}\phi_{u^-}|u^-|^2
$$
does not hold in general for $u\in H^1(\R^3)$. To overcome this
difficulty, we adopt an idea from \cite{Wang} to construct an
auxiliary operator $A$ (See Section 2), which is the starting point
in constructing a pseudo-gradient vector field guaranteeing
existence of the desired invariant sets of the flow. Since $f\in
C(\R,\R)$ and $A$ is merely continuous, $A$ itself can not be used
to define the flow. Instead, $A$ is used in a similar way to
\cite{Liu4} to construct a locally Lipschitz continuous operator $B$
inheriting the main properties of $A$, and we use $B$ to define the
flow. Finally, by minimax arguments in the presence of invariant
sets we obtain the existence of sign-changing solutions to
(\ref{q1}), proving Theorem \ref{Th1}. For the proof of Theorem
\ref{Th2} the above framework is not directly applicable due to
changes of geometric nature of the variational formulation. We use a
perturbation approach by adding a term growing faster than monomial
of degree $4$ with a small coefficient $\lambda>0$. For the
perturbed problems we apply the program above to establish the
existence of multiple sign-changing solutions, and a convergence
argument allows us to pass limit to the original system.


The paper is organized as follows. Section 2 contains the
variational framework of our problem and some preliminary properties
of $\phi_u$. Section 3 is devoted to the proof of Theorem \ref{Th1}.
In Section 4, we use a perturbation approach to prove Theorem
\ref{Th2}.\vskip0.1in


\s{Preliminaries and functional setting}
\renewcommand{\theequation}{2.\arabic{equation}}

In this paper, we make use of the following notations.
\begin{itemize}
\item [$\bullet$] $\|u\|_p:=\big(\int_{\R^3}|u|^p\big)^{1/p}$ for $p\in
[2,\infty)$ and $u\in L^p(\R^3)$;
\item [$\bullet$] $\|u\|:=\big(\|u\|_2^2+\|\nabla u\|_2^2\big)^{1/2}$ for $u\in
H^1(\R^3)$;
\item [$\bullet$] $C,C_j$ denote (possibly different) positive
constants.
\end{itemize}

For any given $u\in H^1(\R^3)$, the Lax-Milgram theorem implies that
there exists a unique $\phi_u\in \mathcal{D}^{1,2}(\R^3)$ such that
$-\DD\phi_u=u^2$. It is well known that
\begin{equation*}
\lab{t1} \phi_u(x)=\int_{\R^3}\frac{u^2(y)}{4\pi|x-y|}dy.
\end{equation*}
We now summarize some properties of $\phi_u$, which will be used
later. See, for instance, \cite{Ruiz} for a proof.

\begin{lemma}\lab{l0} \noindent
\begin{itemize}
\item [{\rm(1)}] $\phi_u(x)\ge 0,\ x\in\R^3$;
\item [{\rm(2)}] there exists $C>0$ independent of $u$ such that
$$
\int_{\R^3}\phi_uu^2\le C\|u\|^4;
$$
\item [{\rm(3)}] if $u$ is a radial function, then so is $\phi_u$;
\item [{\rm(4)}] if $u_n\rg u$ strongly in $L^{\frac{12}{5}}(\R^3)$,
then $\phi_{u_n}\rg \phi_u$ strongly in $\mathcal{D}^{1,2}(\R^3)$.
\end{itemize}
\end{lemma}

Define the Sobolev space
$$
E=\left\{u\in\mathcal{D}^{1,2}(\R^3): \int_{\R^3}V(x)u^2<\iy\right\}
$$
with the norm
$$
\|u\|_E=\left(\int_{\R^3}\big(|\na u|^2+V(x)u^2\big)\right)^\frac12.
$$
This is a Hilbert space and its inner product is denoted by
$(\cdot,\cdot)_E$.

\begin{remark}\lab{r1} {\rm By $(V_0)$,
the embedding $E\hookrightarrow L^q(\R^3)\ (2\le q<6)$ is compact.
This fact implies the (PS) condition; see, e.g., \cite{BW1}. As in
\cite{BPW}, $(V_0)$ can be replaced with the weaker condition:
\begin{itemize}
\item [$(V_0)'$] There exists $r>0$ such that for any $b>0$,
$$\lim\limits_{|y|\rg\iy}m(\{x\in \R^3:V(x)\le b\}\cap B_r(y))=0,$$
where $B_r(y)=\{x\in\R^3:\ |x-y|<r\}$ and $m$ is the Lebesgue
measure in $\R^3$.
\end{itemize}
} \end{remark}

Let us define
$$
D(f,g)=\int_{\R^3}\int_{\R^3}\frac{f(x)g(y)}{4\pi|x-y|}dxdy.
$$
In particular, for $u\in H^1(\R^3)$,
$D(u^2,u^2)=\int_{\R^3}\phi_uu^2$. Moreover, we have the following
properties. For a proof, we refer to \cite[p.250]{Lieb} and
\cite{Ruiz2}.

\begin{lemma}\lab{l1} \noindent
\begin{itemize}
\item [{\rm(1)}] $D(f,g)^2\le D(f,f)D(g,g)$ for any $f,g\in L^{\frac{6}{5}}(\R^3)$;
\item [{\rm(2)}] $D(uv,uv)^2\le D(u^2,u^2)D(v^2,v^2)$
for any $u,v\in L^{\frac{12}{5}}(\R^3)$.
\end{itemize}
\end{lemma}

Substituting $\phi=\phi_u$ into system (\ref{q1}), we can rewrite
system (\ref{q1}) as the single equation
\begin{equation}\lab{q2}
-\DD u+V(x)u+\phi_uu=f(u),\ \ u\in E.
\end{equation}
We define the energy functional $I$ on $E$ by
$$
I(u)=\frac{1}{2}\int_{\R^3}\big(|\na
u|^2+V(x)u^2\big)+\frac{1}{4}\int_{\R^3}\phi_uu^2-\int_{\R^3}F(u).
$$
It is standard to show that $I\in C^1(E,\R)$ and
$$
\langle I'(u),v\rangle=\int_{\R^3}\big(\na u\cdot\na
v+V(x)uv+\phi_uuv-f(u)v\big),\quad u,\ v\in E.
$$
It is easy to verify that $(u,\phi_u)\in
E\times\mathcal{D}^{1,2}(\R^3)$ is a solution of (\ref{q1}) if and
only if $u\in E$ is a critical point of $I$. \vskip0.1in


\s{Proof of Theorem \ref{Th1}}
\renewcommand{\theequation}{3.\arabic{equation}}

In this section, we prove the existence of sign-changing solutions
to system (\ref{q1}) in the case $\mu>4$, working with (\ref{q2}).

\subsection{Properties of operator $A$}

We introduce an auxiliary operator $A$, which will be used to
construct the descending flow for the functional $I$. Precisely, the
operator $A$ is defined as follows: for any $u\in E$, $v=A(u)\in E$
is the unique solution to the equation
\begin{equation}\lab{a1}
-\DD v+V(x)v+\phi_uv=f(u),\ \ \ v\in E.
\end{equation}
Clearly, the three statements are equivalent: $u$ is a solution of
\eqref{q2}, $u$ is a critical point of $I$, and $u$ is a fixed point
of $A$.

\begin{lemma}\lab{l3.1}
The operator $A$ is well defined and is continuous and compact.
\end{lemma}

\begin{proof} Let $u\in E$ and define
$$
J_0(v)=\frac{1}{2}\int_{\R^3}\big(|\na
v|^2+(V(x)+\phi_u)v^2\big)-\int_{\R^3}f(u)v,\quad v\in E.
$$
Then $J_0\in C^1(E,\R)$. By $(f_1)$-$(f_2)$ and Remark \ref{r1},
$J_0$ is coercive, bounded below, weakly lower semicontinuous, and
strictly convex. Thus, $J_0$ admits a unique minimizer $v=A(u)\in
E$, which is the unique solution to \eqref{a1}. Moreover, $A$ maps
bounded sets into bounded sets.

In the following, we prove that $A$ is continuous. Let
$\{u_n\}\subset E$ with $u_n\rg u\in E$ strongly in $E$. Let
$v=A(u)$ and $v_n=A(u_n)$. We need to prove $\|v_n-v\|_E\to0$. We
have
\begin{align*}
\|v-v_n\|_E^2&=\int_{\R^3}(\phi_{u_n}v_n-\phi_uv)(v-v_n)
    +\int_{\R^3}(f(u)-f(u_n))(v-v_n)\\
&=I_1+I_2.
\end{align*}
By Lemma \ref{l0} and Lemma \ref{l1},
\begin{align*}
I_1\le&\int_{\R^3}(\phi_{u_n}v-\phi_uv)(v-v_n)\\
=&D(u_n^2-u^2,v(v-v_n))\\
\le& D(u_n^2-u^2,u_n^2-u^2)^{\frac{1}{2}}D(v(v-v_n),v(v-v_n))^{\frac{1}{2}}\\
\le& D((u_n-u)^2,(u_n-u)^2)^{\frac{1}{4}}D((u_n+u)^2,(u_n+u)^2)^{\frac{1}{4}}\\
&\times D(v^2,v^2)^{\frac{1}{4}}D((v-v_n)^2,(v-v_n)^2)^{\frac{1}{4}}\\
\le& C_1\|u_n-u\|\|u_n+u\|\|v\|\|v-v_n\|\\
\le& C_1\|u_n-u\|_E\|v-v_n\|_E.
\end{align*}
Now, we estimate the second term $I_2$. Let $\phi\in C_0^\iy(\R)$ be
such that $\phi(t)\in[0,1]$ for $t\in\R$, $\phi(t)=1$ for $|t|\le 1$
and $\phi(t)=0$ for $|t|\ge 2$. Setting
$$g_1(t)=\phi(t)f(t),\ \ g_2(t)=f(t)-g_1(t).$$ By $(f_1)$-$(f_2)$, there
exists $C_2>0$ such that $|g_1(s)|\le C_2|s|$ and $|g_2(s)|\le
C_2|s|^5$ for $s\in \R$. Then,
\begin{align*}
I_2=&\int_{\R^3}(g_1(u)-g_1(u_n))(v-v_n)+\int_{\R^3}(g_2(u)-g_2(u_n))(v-v_n)\\
\le&\left(\int_{\R^3}|g_1(u_n)-g_1(u)|^2\right)^{\frac{1}{2}}
\left(\int_{\R^3}|v-v_n|^2\right)^{\frac{1}{2}}\\
&+\left(\int_{\R^3}|g_2(u_n)-g_2(u)|^{\frac{6}{5}}\right)^{\frac{5}{6}}
\left(\int_{\R^3}|v-v_n|^6\right)^{\frac{1}{6}}\\
\le&
C_3\|v-v_n\|_E\left[\left(\int_{\R^3}|g_1(u_n)-g_1(u)|^2\right)^{\frac{1}{2}}
+\left(\int_{\R^3}|g_2(u_n)-g_2(u)|^{\frac{6}{5}}\right)^{\frac{5}{6}}\right].
\end{align*}
Thus,
\begin{align*}
\|v-v_n\|_E\le& C_4\left[\|u-u_n\|_E
+\left(\int_{\R^3}|g_1(u_n)-g_1(u)|^2\right)^{\frac{1}{2}}\right.\\
&
\left.+\left(\int_{\R^3}|g_2(u_n)-g_2(u)|^{\frac{6}{5}}\right)^{\frac{5}{6}}\right].
\end{align*}
Therefore, by the dominated convergence theorem, $\|v-v_n\|_E\rg 0$
as $n\rg\iy$.

Finally, we show that $A$ is compact. Let $\{u_n\}\subset E$ be a
bounded sequence. Then $\{v_n\}\subset E$ is a bounded sequence,
where, as above, $v_n=A(u_n)$. Passing to a subsequence, by Remark
\ref{r1}, we may assume that $u_n\rg u$ and $v_n\rg v$ weakly in $E$
and strongly in $L^q(\R^3)$ as $n\rg\iy$ for $q\in[2,6)$. Consider
the identity
\begin{equation}\label{identity}
\int_{\mathbb R^3}\big(\nabla
v_n\cdot\nabla\xi+Vv_n\xi+\phi_{u_n}v_n\xi\big)=\int_{\mathbb
R^3}f(u_n)\xi,\quad \xi\in E.
\end{equation}
Since $u_n\to u$ strongly in $L^\frac{12}{5}(\mathbb R^3)$, it
follows from Lemma \ref{l0}(4) and the Sobolev imbedding theorem
that $\phi_{u_n}\to\phi_u$ strongly in $L^6(\mathbb R^3)$. Since, in
addition, $v_n\to v$ strongly in $L^\frac{12}{5}(\mathbb R^3)$,
using the H\"older inequality, we have
$$
\left|\int_{\mathbb R^3}(\phi_{u_n}v_n-\phi_uv)\xi\right|
\leq\|\phi_{u_n}\|_6\|v_n-v\|_\frac{12}{5}\|\xi\|_\frac{12}{5}
+\|\phi_{u_n}-\phi_u\|_6\|v\|_\frac{12}{5}\|\xi\|_\frac{12}{5}\to0
$$
for any $\xi\in E$. Taking limit as $n\to\infty$ in \eqref{identity}
yields
\begin{equation*}
\int_{\mathbb R^3}\big(\nabla
v\cdot\nabla\xi+Vv\xi+\phi_{u}v\xi\big)=\int_{\mathbb
R^3}f(u)\xi,\quad \xi\in E.
\end{equation*}
This means $v=A(u)$ and thus
\begin{align*}
\|v-v_n\|_E^2=\int_{\R^3}\big(\phi_uv(v_n-v)-\phi_{u_n}v_n(v_n-v)\big)
+\int_{\R^3}(f(u_n)-f(u))(v_n-v).
\end{align*}
Hence, in the same way as above, $\|v-v_n\|_E\rg 0$, i.e.,
$A(u_n)\rg A(u)$ in $E$ as $n\rg\iy$.
\end{proof}

\begin{remark} {\rm Obviously, if $f$ is odd then $A$ is odd.} \end{remark}

\begin{lemma}\lab{l3.2} \noindent
\begin{itemize}
\item [{\rm (1)}] $\langle I'(u),u-A(u)\rangle\ge \|u-A(u)\|_E^2$ for all $u\in E$;
\item [{\rm (2)}] $\|I'(u)\|\le \|u-A(u)\|_E(1+C\|u\|_E^2)$ for some $C>0$
and all $u\in E$.
\end{itemize}
\end{lemma}

\begin{proof}
Since $A(u)$ is the solution of equation (\ref{a1}), we see that
\begin{equation}\lab{a2} \langle
I'(u),u-A(u)\rangle=\|u-A(u)\|_E^2+\int_{\R^3}\phi_u(u-A(u))^2,
\end{equation}
which implies $\langle I'(u),u-A(u)\rangle\ge \|u-A(u)\|_E^2$ for
all $u\in E$. For any $\varphi\in E$, we have
\begin{align*}
\langle
I'(u),\varphi\rangle&=(u-A(u),\varphi)_E+\int_{\R^3}\phi_u(u-A(u))\varphi\\
&=(u-A(u),\varphi)_E+D(u^2,(u-A(u))\varphi).
\end{align*}
By Lemma \ref{l0} and Lemma \ref{l1}, $$|D(u^2,(u-A(u))\varphi)|\le
C\|u\|_E^2\|u-A(u)\|_E\|\varphi\|_E.$$ Thus, $\|I'(u)\|\le
\|u-A(u)\|_E(1+C\|u\|_E^2)$ for all $u\in E$. \end{proof}

\begin{lemma}\lab{l3.3}
For $a<b$ and $\alpha>0$, there exists $\beta>0$ such that
$\|u-A(u)\|_E\ge\beta$ if $u\in E$, $I(u)\in [a,b]$ and
$\|I'(u)\|\ge\alpha$.
\end{lemma}

\begin{proof} For $u\in E$, by $(f_3)$, we have
\begin{align*}
I(u)&-\frac{1}{\mu}(u,u-A(u))_E\\
=&\left(\frac{1}{2}-\frac{1}{\mu}\right)\|u\|_E^2
  +\left(\frac{1}{4}-\frac{1}{\mu}\right)\int_{\R^3}\phi_uu^2\\
&+\frac{1}{\mu}\int_{\R^3}\phi_uu(u-A(u))
+\int_{\R^3}\big(\frac{1}{\mu}f(u)u-F(u)\big)\\
\ge&\left(\frac{1}{2}-\frac{1}{\mu}\right)\|u\|_E^2
  +\left(\frac{1}{4}-\frac{1}{\mu}\right)\int_{\R^3}\phi_uu^2
  +\frac{1}{\mu}\int_{\R^3}\phi_uu(u-A(u)).
\end{align*}
Then,
\begin{equation}
\lab{a2} \|u\|_E^2+\int_{\R^3}\phi_uu^2\le
C_1\left(|I(u)|+\|u\|_E\|u-A(u)\|_E+\left|\int_{\R^3}\phi_uu(u-A(u))\right|\right).
\end{equation}
By H\"{o}lder's inequality and Lemmas \ref{l0} and \ref{l1},
\begin{align*}
\left|\int_{\R^3}\phi_uu(u-A(u))\right|&\le
\left(\int_{\R^3}\phi_u(u-A(u))^2\right)^{\frac{1}{2}}
\left(\int_{\R^3}\phi_uu^2\right)^{\frac{1}{2}}\\
&\le
C_2\|u\|_E\|u-A(u)\|_E\left(\int_{\R^3}\phi_uu^2\right)^{\frac{1}{2}}.
\end{align*}
Thus, it follows from (\ref{a2}) that
\begin{equation}\lab{a3}
\|u\|_E^2\le
C_3\left(|I(u)|+\|u\|_E\|u-A(u)\|_E+\|u\|_E^2\|u-A(u)\|_E^2\right).
\end{equation}
If there exists $\{u_n\}\subset E$ with $I(u_n)\in [a,b]$ and
$\|I'(u_n)\|\ge\alpha$ such that $\|u_n-A(u_n)\|_E\rg 0$ as
$n\rg\iy$, then it follows from (\ref{a3}) that $\{\|u_n\|_E\}$ is
bounded, and by Lemma \ref{l3.2} we see that $\|I'(u_n)\|\rg 0$ as
$n\rg\iy$, which is a contradiction. Thus, the proof is completed.
\end{proof}

\subsection{Invariant subsets of descending flow}
To obtain sign-changing solutions, we make use of the positive and
negative cones as in many references such as \cite{Liu2, Liu4, BW2,
Wang}. Precisely, define
$$
P^+:=\{u\in E:u\ge 0\}\ \ \mbox{and}\ \ P^-:=\{u\in E:u\le 0\}.
$$
Set for $\e>0$,
$$
P_\e^+:=\{u\in E: \mbox{dist}(u,P^+)<\e\}\ \ \mbox{and}\ \ P_\e^-:=\{u\in E:
\mbox{dist}(u,P^-)<\e\},
$$
where $\mbox{dist}(u,P^\pm)=\inf\limits_{v\in P^\pm}\|u-v\|_E$.
Obviously, $P_\e^-=-P_\e^+$. Let $W=P_\e^+\cup P_\e^-$. Then $W$ is
an open and symmetric subset of $E$ and $E\setminus W$ contains only
sign-changing functions. On the other hand, the next lemma shows
that, for $\e$ small, all sign-changing solutions to \eqref{q2} are
contained in $E\setminus W$.

\begin{lemma}\lab{l4.1}
There exists $\e_0>0$ such that for $\e\in(0,\e_0)$,
\begin{itemize}
\item [{\rm(1)}] $A(\pa P_\e^-)\subset P_\e^-$ and every nontrivial solution $u\in
P_\e^-$ is negative,
\item [{\rm(2)}] $A(\pa P_\e^+)\subset P_\e^+$ and every nontrivial solution $u\in
P_\e^+$ is positive.
\end{itemize}
\end{lemma}

\begin{proof}
Since the two conclusions are similar, we only prove the first one.
By $(f_1)$-$(f_2)$, for any fixed $\delta>0$, there exists
$C_\delta>0$ such that
$$
|f(t)|\le \delta|t|+C_\delta|t|^p,\ \ t\in\R.
$$
Let $u\in E$ and $v=A(u)$. By Remark \ref{r1}, for any $q\in[2,6]$,
there exists $m_q>0$ such that
\begin{equation}\label{123}
\|u^\pm\|_q=\inf\limits_{w\in P^\mp}\|u-w\|_q\le
m_q\inf\limits_{w\in P^\mp}\|u-w\|_E=m_q\mbox{dist}(u,P^\mp).
\end{equation}
Obviously, $\mbox{dist}(v,P^-)\le\|v^+\|_E$. Then, by $(f_3)$, we
estimate
\begin{align*}
\mbox{dist}(v,P^-)\|v^+\|_E&\le\|v^+\|_E^2=(v,v^+)_E\\
&=\int_{\R^3}\big(f(u)v^+-\phi_uvv^+\big)\\
&\le\int_{\R^3}f(u)v^+\le\int_{\R^3}f(u^+)v^+\\
&\le\int_{\R^3}(\delta|u^+|+C_\delta|u^+|^{p-1})|v^+|\\
&\le \delta\|u^+\|_2\|v^+\|_2+C_\delta\|u^+\|_p^{p-1}\|v^+\|_{p}\\
&\le
C\left(\delta\mbox{dist}(u,P^-)+C_\delta\mbox{dist}(u,P^-)^{p-1}\right)\|v^+\|_E.
\end{align*}
It follows that
$$
\mbox{dist}(A(u),P^-)\le
C\left(\delta\mbox{dist}(u,P^-)+C_\delta\mbox{dist}(u,P^-)^{p-1}\right).
$$
Thus, choosing $\delta$ small enough, there exists $\e_0>0$ such
that for $\e\in(0,\e_0)$,
$$\mbox{dist}(A(u),P^-)\le \frac{1}{2}\mbox{dist}(u,P^-)\ \ \mbox{for
any}\ \  u\in P_\e^-.$$ This implies that $A(\pa P_\e^-)\subset
P_\e^-$. If there exists $u\in P_\e^-$ such that $A(u)=u$, then
$u\in P^-$. If $u\not\equiv 0$, by the maximum principle, $u<0$ in
$\R^3$.
\end{proof}

Denote the set of fixed points of $A$ by $K$, which is exactly the
set of critical points of $I$. Since $A$ is merely continuous, $A$
itself is not applicable to construct a descending flow for $I$, and
we have to construct a locally Lipschitz continuous operator $B$ on
$E_0:=E\setminus K$ which inherits the main properties of $A$.

\begin{lemma}\lab{l4.2}
There exists a locally Lipschitz continuous operator $B:E_0\to E$
such that
\begin{itemize}
\item [{\rm(1)}] $B(\pa P_\e^+)\subset P_\e^+$ and $B(\pa P_\e^-)\subset P_\e^-$
for $\e\in(0,\e_0)$;
\item [{\rm(2)}] $\frac{1}{2}\|u-B(u)\|_E\le\|u-A(u)\|_E\le2\|u-B(u)\|_E$ for
all $u\in E_0$;
\item [{\rm(3)}] $\langle I'(u),u-B(u)\rangle\ge \frac{1}{2}\|u-A(u)\|_E^2$
for all $u\in E_0$;
\item [{\rm(4)}] if $f$ is odd then $B$ is odd.
\end{itemize}
\end{lemma}

\begin{proof}
The proof is similar to the proofs of \cite[Lemma 4.1]{Liu3} and
\cite[Lemma 2.1]{Liu4}. We omit the details.
\end{proof}

\subsection{Existence of one sign-changing solution}

In this subsection, we will find one sign-changing solution of
(\ref{q2}) via mini-max method incorporated with invariant sets of
descending flow. First of all, we introduce the critical point
theorem \cite[Theorem 2.4]{Wang}. For more details, we refer to
\cite{Wang}.

Let $X$ be a Banach space, $J\in C^1(X,\R)$, $P,Q\subset X$ be open
sets, $M=P\cap Q$, $\Sg=\pa P\cap\pa Q$ and $W=P\cup Q$. For
$c\in\R$, $K_c=\{x\in X: J(x)=c, J'(x)=0\}$ and $J^c=\{x\in X:
J(x)\le c\}$. In \cite{Wang}, a critical point theory on metric
spaces was given, but here we only need a Banach space version of
the theory.

\begin{definition}
{\rm (\cite{Wang}) $\{P,Q\}$ is called an admissible family of
invariant sets with respect to $J$ at level $c$, provided that the
following deformation property holds: if $K_c\setminus W=\emptyset$,
then, there exists $\e_0>0$ such that for $\e\in(0,\e_0)$, there
exists $\eta\in C(X,X)$ satisfying
\begin{itemize}
\item [(1)] $\eta(\overline{P})\subset \overline{P}$, $\eta(\overline{Q})\subset
\overline{Q}$;
\item [(2)] $\eta\mid_{J^{c-2\e}}=id$;
\item [(3)] $\eta(J^{c+\e}\setminus W)\subset J^{c-\e}$.
\end{itemize}}
\end{definition}

\noindent{\bf Theorem A.} (\cite{Wang}) {\it Assume that $\{P,Q\}$
is an admissible family of invariant sets with respect to $J$ at any
level $c\ge c_{\ast}:=\inf_{u\in\Sigma}J(u)$ and there exists a map
$\vp_0:\Delta\rg X$ satisfying
\begin{itemize}
\item [{\rm(1)}] $\vp_0(\pa_1\Delta)\subset P$ and $\vp_0(\pa_2\Delta)\subset
      Q$,
\item [{\rm(2)}] $\vp_0(\pa_0\Delta)\cap M=\emptyset$,
\item [{\rm(3)}] $\sup\limits_{u\in\vp_0(\pa_0\Delta)}J(u)<c_\ast$,
\end{itemize}
where $\Delta=\{(t_1,t_2)\in\R^2:t_1,t_2\ge 0,\ t_1+t_2\le1\}$,
$\pa_1\Delta=\{0\}\times [0,1]$, $\pa_2\Delta=[0,1]\times\{0\}$ and
$\pa_0\Delta=\{(t_1,t_2)\in\R^2:t_1,t_2\ge 0,\ t_1+t_2=1\}$. Define
$$
c=\inf\limits_{\varphi\in\G}\sup\limits_{u\in\varphi(\triangle)\setminus
W}J(u),
$$
where $\G:=\left\{\varphi\in
C(\triangle,X):\varphi(\pa_1\triangle)\subset P,\
\varphi(\pa_2\triangle)\subset Q,\
\varphi|_{\pa_0\triangle}=\varphi_0|_{\pa_0\triangle}\right\}.$ Then
$c\ge c_\ast$ and $K_c\setminus W\not=\emptyset$. }

Now, we use Theorem A to prove the existence of a sign-changing
solution to problem ({\ref{q2}}), and for this we take $X=E$,
$P=P_\e^+$, $Q=P_\e^-$ and $J=I$. We will show that
$\{P_\e^+,P_\e^-\}$ is an admissible family of invariant sets for
the functional $I$ at any level $c\in\R$. Indeed, if $K_c\setminus
W=\emptyset$, then $K_c\subset W$. Since $\mu>4$, by Remark
\ref{r1}, it is easy to see that $I$ satisfies the (PS)-condition
and therefore $K_c$ is compact. Thus, $2\delta:=\mbox{dist}(K_c,\pa
W)>0$.


\begin{lemma}\lab{l5.1}
If $K_c\setminus W=\emptyset$, then there exists $\e_0>0$ such that,
for $0<\e<\e'<\e_0$, there exists a continuous map $\si:[0,1]\times
E\rightarrow E$ satisfying
\begin{itemize}
\item [{\rm(1)}] $\si(0,u)=u$ for $u\in E$;
\item [{\rm(2)}] $\si(t,u)=u$ for $t\in[0,1]$, $u\not\in
I^{-1}[c-\e',c+\e']$;
\item [{\rm(3)}] $\si(1,I^{c+\e}\setminus W)\subset I^{c-\e}$;
\item [{\rm(4)}] $\si(t,\overline{P_\e^+})\subset \overline{P_\e^+}$
and $\si(t,\overline{P_\e^-})\subset\overline{P_\e^-}$ for
$t\in[0,1]$.
\end{itemize}
\end{lemma}

\begin{proof}
The proof is similar to the proof of \cite[Lemma 3.5]{Wang}. For the
sake of completeness, we give the details here. For $G\subset E$ and
$a>0$, let $N_a(G):=\{u\in E: \mbox{dist}(u,G)<a\}$. Then
$N_\delta(K_c)\subset W$. Since $I$ satisfies the (PS)-condition,
there exist $\e_0, \alpha>0$ such that
$$
\|I'(u)\|\ge\alpha\ \ \mbox{for}\ \ u\in
I^{-1}([c-\e_0,c+\e_0])\setminus N_\frac{\delta}{2}(K_c).
$$
By Lemmas \ref{l3.3} and \ref{l4.2}, there exists $\beta>0$ such
that
$$
\|u-B(u)\|_E\ge\beta\ \ \mbox{for}\ \ u\in
I^{-1}([c-\e_0,c+\e_0])\setminus N_\frac{\delta}{2}(K_c).
$$
Without loss of generality, assume that $\e_0\le
\frac{\beta\dd}{32}$. Let
$$
V(u)=\frac{u-B(u)}{\|u-B(u)\|_E}\ \ \ \mbox{for}\ \ u\in
E_0=E\setminus K,
$$
and take a cut-off function $g: E\rightarrow [0,1]$, which is
locally Lipschitz continuous, such that
\begin{eqnarray*}
g(u)=\left\{
\begin{array}{lll}
0,\ \mbox{if}\ \ u\not\in I^{-1}[c-\e',c+\e']\ \ \mbox{or}\ \
u\in N_{\frac{\dd}{4}}(K_c),\\
1,\ \mbox{if}\ \ u\in I^{-1}[c-\e,c+\e]\ \ \mbox{and}\ \ u\not\in
N_{\frac{\dd}{2}}(K_c).
\end{array}
\right.
\end{eqnarray*}
Decreasing $\varepsilon_0$ if necessary, one may find a $\nu>0$ such
that $I^{-1}[c-\e_0,c+\e_0]\cap N_\nu(K)\subset N_{\delta/4}(K_c)$,
and this can be seen as a consequence of the (PS) condition. Thus,
$g(u)=0$ for any $u\in N_\nu(K)$. By Lemma \ref{l4.2},
$g(\cdot)V(\cdot)$ is locally Lipschitz continuous on $E$.

Consider the following initial value problem
\begin{equation}\lab{a5} \left\{
\begin{array}{lll}
\frac{d\tau}{dt}=-g(\tau)V(\tau),\\
\tau(0,u)=u.
\end{array}
\right. \end{equation} For any $u\in E$, one sees that problem
(\ref{a5}) admits a unique solution $\tau(\cdot,u)\in C(\R^+,E)$.
Define $\si(t,u)=\tau(\frac{16\e}{\beta}t,u)$. It suffices to check
(3) and (4) since $(1)$ and $(2)$ are obvious.

To verify $(3)$ we let $u\in I^{c+\e}\setminus W$. By Lemma
\ref{l4.2}, $I(\tau(t,u))$ is decreasing for $t\ge 0$. If there
exists $t_0\in [0,\frac{16\e}{\beta}]$ such that
$I(\tau(t_0,u))<c-\e$ then
$I(\si(1,u))=I(\tau(\frac{16\e}{\beta},u))<c-\e$. Otherwise, for any
$t\in[0,\frac{16\e}{\beta}]$, $I(\tau(t,u))\ge c-\e$. Then,
$\tau(t,u)\in I^{-1}[c-\e,c+\e]$ for $t\in[0,\frac{16\e}{\beta}]$.
We claim that for any $t\in[0,\frac{16\e}{\beta}]$,
$\tau(t,u)\not\in N_{\frac{\dd}{2}}(K_c)$. If, for some
$t_1\in[0,\frac{16\e}{\beta}]$, $\tau(t_1,u)\in
N_{\frac{\dd}{2}}(K_c)$, then, since $u\not\in N_\delta(K_c)$,
$$
\frac{\dd}{2}\le\|\tau(t_1,u)-u\|_E\le\int_0^{t_1}\|\tau'(s,u)\|_Eds\le
t_1\le \frac{16\e}{\beta},
$$
which contradicts the fact that
$\varepsilon<\varepsilon_0\leq\frac{\beta\delta}{32}$. So
$g(\tau(t,u))\equiv1$ for $t\in[0,\frac{16\e}{\beta}]$. Then by (2)
and (3) of Lemma \ref{l4.2},
\begin{align*}
I(\tau(\frac{16\e}{\beta},u))&=I(u)-\int_0^{\frac{16\e}{\beta}}\langle
I'(\tau(s,u)),V(\tau(s,u))\rangle\\
&\le
I(u)-\int_0^{\frac{16\e}{\beta}}\frac{1}{8}\|\tau(s,u)-B\tau(s,u)\|_E\\
&\le c+\e-\frac{16\e}{\beta}\frac{\beta}{8}=c-\e.
\end{align*}
Finally, $(4)$ is a consequence of (1) of Lemma \ref{l4.2} (see
\cite{Liu1} for a detailed proof).
\end{proof}

\begin{corollary}\lab{cor1}
$\{P_\e^+,P_\e^-\}$ is an admissible family of invariant sets for
the functional $I$ at any level $c\in\R$.
\end{corollary}

\begin{proof}
The conclusion follows from Lemma \ref{l5.1}.
\end{proof}

In the following, we will construct $\vp_0$ satisfying the
hypotheses in Theorem A. Choose $v_1,v_2\in C_0^\infty(\mathbb
R^3)\setminus\{0\}$ satisfying
$\mbox{supp}(v_1)\cap\mbox{supp}(v_2)=\emptyset$ and $v_1\le
0,v_2\ge 0$. Let $\vp_0(t,s):=R(tv_1+sv_2)$ for $(t,s)\in\Delta$,
where $R$ is a positive constant to be determined later. Obviously,
for $t,s\in[0,1]$, $\vp_0(0,s)=Rsv_2\in P_\e^+$ and
$\vp_0(t,0)=Rtv_1\in P_\e^-$.

\begin{lemma}\lab{l5.2}
For $q\in[2,6]$ there exists $m_q>0$ independent of $\e$ such that
$\|u\|_q\le m_q\e$ for $u\in M=P_\e^+\cap P_\e^-$.
\end{lemma}

\begin{proof} This follows from \eqref{123}.
\end{proof}

\begin{lemma}\lab{l5.3}
If $\e>0$ is small enough then $I(u)\ge\frac{\e^2}{2}$ for
$u\in\Sigma=\pa P_\e^+\cap\pa P_\e^-$, that is,
$c_\ast\ge\frac{\e^2}{2}$.
\end{lemma}

\begin{proof} For $u\in\pa P_\e^+\cap\pa P_\e^-$, we have
$\|u^\pm\|_E\ge\mbox{dist}(u,P^\mp)=\e$. By $(f_1)$-$(f_2)$, we have
$F(t)\le\frac{1}{3m_2^2}|t|^2+C_1|t|^{p}$ for all $t\in\R$. Then,
using Lemma \ref{l5.2}, we see that
$$
I(u)\ge\e^2-\frac{1}{3}\e^2-C_2\e^{p}\ge\frac{\e^2}{2},
$$
for $\e$ small enough.
\end{proof}

\begin{proof}[{\bf Proof of Theorem \ref{Th1}
(Existence part)}] It suffices to verify assumptions $(2)$-$(3)$ in
applying Theorem A. Observe that $\rho=\min\{\|tv_1+(1-t)v_2\|_2:\
0\leq t\leq1\}>0$. Then, $\|u\|_2\geq\rho R$ for
$u\in\vp_0(\pa_0\Delta)$ and it follows from Lemma \ref{l5.2} that
$\vp_0(\pa_0\Delta)\cap M=\emptyset$ for $R$ large enough. By
$(f_3)$, we have $F(t)\ge C_1|t|^\mu-C_2$ for any $t\in\R$. For any
$u\in \vp_0(\pa_0\Delta)$, by Lemma \ref{l0},
\begin{align*}
I(u)&\le\frac{1}{2}\|u\|_E^2+C_3\|u\|_E^4
-\int_{\mbox{supp}(v_1)\cup\mbox{supp}(v_2)}F(u)\\
&\le\frac{1}{2}\|u\|_E^2+C_3\|u\|_E^4-C_1\|u\|_\mu^\mu+C_4,
\end{align*}
which together with Lemma \ref{l5.3} implies that, for $R$ large
enough and $\e$ small enough,
$$\sup\limits_{u\in\vp_0(\pa_0\Delta)}I(u)<0<c_\ast.$$
According to Theorem A, $I$ has at least one critical point $u$ in
$E\setminus(P_\e^+\cup P_\e^-)$, which is a sign-changing solution
of equation (\ref{q2}). Then $(u,\phi_u)$ is a sign-changing
solution of system \eqref{q1}.
\end{proof}

\subsection{Existence of infinitely many sign-changing solutions}

In this subsection, we prove the existence of infinitely many
sign-changing solutions to system (\ref{q1}). For this we will make
use of \cite[Theorem 2.5]{Wang}, which we recall below.


We will use the notations from Subsection 3.3. Assume $G:X\rg X$ to
be an isometric involution, that is, $G^2=id$ and $d(Gx,Gy)=d(x,y)$
for $x,y\in X$. We assume $J$ is $G-$invariant on $X$ in the sense
that $J(Gx)=J(x)$ for any $x\in X$. We also assume $Q=GP$. A subset
$F\subset X$ is said to be symmetric if $Gx\in F$ for any $x\in F$.
The genus of a closed symmetric subset $F$ of $X\setminus\{0\}$ is
denoted by $\gamma(F)$.

\bd\lab{wa} {\rm(\cite{Wang}) $P$ is called a $G-$admissible
invariant set with respect to $J$ at level $c$, if the following
deformation property holds: there exist $\e_0>0$ and a symmetric
open neighborhood $N$ of $K_c\setminus W$ with
$\gamma(\overline{N})<\iy$, such that for $\e\in(0,\e_0)$ there
exists $\eta\in C(X,X)$ satisfying
\begin{itemize}
\item [(1)] $\eta(\overline{P})\subset \overline{P}$, $\eta(\overline{Q})\subset
\overline{Q}$;
\item [(2)] $\eta\circ G=G\circ\eta$;
\item [(3)] $\eta\mid_{J^{c-2\e}}=id$;
\item [(4)] $\eta(J^{c+\e}\setminus(N\cup W))\subset J^{c-\e}$.
\end{itemize}}
\ed

\noindent{\bf Theorem B.} (\cite{Wang}) {\it Assume that $P$ is a
$G$-admissible invariant set with respect to $J$ at any level $c\ge
c^\ast:=\inf_{u\in\Sg}J(u)$ and for any $n\in \mathbb{N}$, there
exists a continuous map $\varphi_n:B_n:=\{x\in\R^n:|x|\le
1\}\rightarrow X$ satisfying
\begin{itemize}
\item [{\rm(1)}] $\varphi_n(0)\in M:=P\cap Q$,
$\varphi_n(-t)=G\varphi_n(t)$ for $t\in B_n$,
\item [{\rm(2)}] $\varphi_n(\pa B_n)\cap M=\emptyset$,
\item [{\rm(3)}] $\sup_{u\in{\rm Fix}_G\cup\varphi_n(\pa
B_n)}J(u)<c^\ast,$ where ${\rm Fix}_G:=\{u\in X:Gu=u\}$.
\end{itemize}
For $j\in\mathbb{N}$, define
$$c_j=\inf\limits_{B\in\G_j}\sup\limits_{u\in B\setminus
W}J(u),$$ where
\begin{align*}
\G_j:=\left\{B\Big|\ \begin{array}{ll} B=\varphi(B_n\setminus Y)
\hbox{ for some $\varphi\in G_n$, $n\geq j$, and open $Y\subset B_n$}\\
\hbox{such that $-Y=Y$ and $\g(\bar{Y})\le n-j$}
\end{array}\right\}
\end{align*}
and
\begin{align*}
G_n:=\left\{\varphi\Big|\ \begin{array}{ll} \varphi\in C(B_n,X),\
\hbox{$\varphi(-t)=G\varphi(t)$ for $t\in B_n$,}\\
\hbox{$\varphi(0)\in M$ and $\varphi|_{\pa B_n}=\varphi_n|_{\pa
B_n}$}
\end{array}\right\}.\end{align*}
Then for $j\ge 2$, $c_j\ge c_\ast$, $K_{c_j}\setminus
W\not=\emptyset$ and $c_j\rg\iy$ as $j\rg\iy$.} \vskip 0.1in

To apply Theorem B, we take $X=E$, $G=-id$, $J=I$ and $P=P_\e^+$.
Then $M=P_\e^+\cap P_\e^-$, $\Sigma=\pa P_\e^+\cap \pa P_\e^-$, and
$W=P_\e^+\cup P_\e^-$. In this subsection, $f$ is assumed to be odd,
and, as a consequence, $I$ is even. Now, we show that $P_\e^+$ is a
$G$-admissible invariant set for the functional $I$ at any level
$c$. Since $K_c$ is compact, there exists a symmetric open
neighborhood $N$ of $K_c\setminus W$ such that
$\gamma(\overline{N})<\iy$.

\begin{lemma}\lab{l6.1} There exists $\e_0>0$ such that for $0<\e<\e'<\e_0$,
there exists a continuous map $\si:[0,1]\times E\rightarrow E$
satisfying
\begin{itemize}
\item [{\rm(1)}] $\si(0,u)=u$ for $u\in E$.
\item [{\rm(2)}] $\si(t,u)=u$ for $t\in[0,1]$, $u\not\in I^{-1}[c-\e',c+\e']$.
\item [{\rm(3)}] $\si(t,-u)=-\si(t,u)$ for $(t,u)\in[0,1]\times E$.
\item [{\rm(4)}] $\si(1,I^{c+\e}\setminus (N\cup W))\subset I^{c-\e}$.
\item [{\rm(5)}] $\si(t,\overline{P_\e^+})\subset \overline{P_\e^+}$,
$\si(t,\overline{P_\e^-})\subset
\overline{P_\e^-}$ for $t\in[0,1]$.
\end{itemize}
\end{lemma}

\begin{proof}
The proof is similar to the proof of Lemma \ref{l5.1}. Since $I$ is
even, $B$ is odd and thus $\sigma$ is odd in $u$.
\end{proof}

\begin{corollary}
$P_\e^+$ is a $G-$admissible invariant set for the functional $I$ at
any level $c$.
\end{corollary}


\begin{proof}[{\bf Proof of Theorem \ref{Th1}
(Multiplicity part)}] According to Theorem B, if $\vp_n$ exists and
satisfies the assumptions in Theorem B then $I$ has infinitely many
critical points in $E\setminus(P_\e^+\cup P_\e^-)$, which are
sign-changing solutions to \eqref{q2} and thus yield sign-changing
solution to \eqref{q1}. It suffices to construct $\vp_n$. For any
$n\in\mathbb{N}$, choose $\{v_i\}_1^n\subset C_0^\infty(\mathbb
R^3)\setminus\{0\}$ such that
$\mbox{supp}(v_i)\cap\mbox{supp}(v_j)=\emptyset$ for $i\not=j$. We
define $\varphi_n\in C(B_n,E)$ as
$$\varphi_n(t)=R_n\sum_{i=1}^nt_iv_i,\quad t=(t_1,t_2,\cdots,t_n)\in B_n,$$
where $R_n>0$. For $R_n$ large enough, it is easy to check that all
the assumptions of Theorem B are satisfied.
\end{proof}

\s{Proof of Theorem \ref{Th2}}
\renewcommand{\theequation}{4.\arabic{equation}}

In this section, we do not assume $\mu>4$ and thus the argument of
Section 3 which essentially depends on the assumption $\mu>4$ is not
valid in the present case. This obstacle will be overcome via a
perturbation approach which is originally due to \cite{Wang}. The
method from Section 3 can be used for the perturbed problem. By
passing to the limit, we then obtain sign-changing solutions of the
original problem (\ref{q1}).

Fix a number $r\in(p,6)$. For any fixed $\la\in(0,1]$, we consider
the modified problem
\begin{equation}\lab{m}
-\DD u+V(x)u+\phi_uu=f(u)+\la|u|^{r-2}u,\quad u\in E
\end{equation}
and its associated functional
$$
I_\la(u)=I(u)-\frac{\la}{r}\int_{\RN}|u|^r.
$$
It is standard to show that $I_\la\in C^1(E,\R)$ and
$$
\langle I_\la'(u),v\rangle=\langle
I'(u),v\rangle-\la\int_{\RN}|u|^{r-2}uv,\quad u,\ v\in E.
$$
For any $u\in E$, we denote by $v=A_\la(u)\in E$ the unique solution
to the problem
$$
-\DD v+V(x)v+\phi_uv=f(u)+\la|u|^{r-2}u,\quad v\in E.
$$
As in Section 3, one verifies that the operator $A_\la:E\rightarrow
E$ is well defined and is continuous and compact. In the following,
if the proof of a result is similar to its counterpart in Section 3,
it will not be written out.

\begin{lemma}\lab{l3.22} \noindent
\begin{itemize}
\item [{\rm (1)}] $\langle I_\la'(u),u-A_\la(u)\rangle\ge \|u-A_\la(u)\|_E^2$
for all $u\in E$;
\item [{\rm (2)}] there exists $C>0$ independent of $\la$ such that
$\|I_\la'(u)\|\le \|u-A_\la(u)\|_E(1+C\|u\|_E^2)$ for all $u\in E$.
\end{itemize}
\end{lemma}

\begin{lemma}\lab{l3.33}
For any $\la\in(0,1)$, $a<b$ and $\alpha>0$, there exists
$\beta(\la)>0$ such that $\|u-A_\la(u)\|_E\ge\beta(\la)$ for any
$u\in E$ with $I_\la(u)\in [a,b]$ and $\|I_\la'(u)\|\ge\alpha$.
\end{lemma}

\begin{proof} Fix a number $\gamma\in(4,q)$. For $u\in E$,
\begin{align*}
&I_\la(u)-\frac{1}{\gamma}(u,u-A_\la(u))_E\\
&=\left(\frac{1}{2}-\frac{1}{\gamma}\right)\|u\|_E^2
+\left(\frac{1}{4}-\frac{1}{\gamma}\right)\int_{\R^3}\phi_uu^2\\
&\ \ \ \ +\frac{1}{\gamma}\int_{\R^3}\phi_uu(u-A_\la(u))
+\int_{\R^3}\big(\frac{1}{\gamma}f(u)u-F(u)\big)
+\la\left(\frac{1}{\gamma}-\frac{1}{r}\right)\|u\|_r^r.
\end{align*}
Then, by $(f_1)$-$(f_2)$,
\begin{align*}&\|u\|_E^2+\int_{\R^3}\phi_uu^2+\la\|u\|_r^r\\
&\ \le
C_1\left(|I_\la(u)|+\|u\|_E\|u-A_\la(u)\|_E+\|u\|_p^p
+\left|\int_{\R^3}\phi_uu(u-A_\la(u))\right|\right).
\end{align*}
Since
\begin{align*} \left|\int_{\R^3}\phi_uu(u-A_\la(u))\right|
\le
C_2\|u\|_E\|u-A_\la(u)\|_E\left(\int_{\R^3}\phi_uu^2\right)^{\frac{1}{2}},
\end{align*}
one sees that
\begin{align}\lab{s01}&\|u\|_E^2+\int_{\R^3}\phi_uu^2+\la\|u\|_r^r\nonumber\\
&\ \le
C_3\left(|I_\la(u)|+\|u\|_p^p+\|u\|_E\|u-A_\la(u)\|_E
+\|u\|_E^2\|u-A_\la(u)\|_E^2\right).
\end{align}
If there exists $\{u_n\}\subset E$ with $I_\la(u_n)\in [a,b]$ and
$\|I_\la'(u_n)\|\ge\alpha$ such that $\|u_n-A_\la(u_n)\|_E\rg 0$ as
$n\rg\iy$, then it follows from (\ref{s01}) that, for large $n$,
\begin{align*}
\|u_n\|_E^2+\int_{\R^3}\phi_{u_n}u_n^2+\la\|u_n\|_r^r\le
C_4(1+\|u_n\|_p^p).
\end{align*}

{\bf Claim: $\{u_n\}$ is bounded in $E$.} Otherwise, assume that
$\|u_n\|_E\rg\iy$ as $n\rg\iy$. Then
\begin{align}\lab{s02}
\|u_n\|_E^2+\int_{\R^3}\phi_{u_n}u_n^2+\la\|u_n\|_r^r\le
C_5\|u_n\|_p^p.
\end{align}
By (\ref{s02}) there exists $C(\la)>0$ such that for large $n$,
$$
\|u_n\|_2^2+\|u_n\|_r^r\le C(\la)\|u_n\|_p^p.
$$
Let $t\in(0,1)$ be such that
$$\frac{1}{p}=\frac{t}{2}+\frac{1-t}{r}.$$
Then, by the interpolation inequality,
$$
\|u_n\|_2^2+\|u_n\|_r^r\le C(\la)\|u_n\|_p^p\le
C(\la)\|u_n\|_2^{tp}\|u_n\|_r^{(1-t)p},
$$
from which it follows that there exist $C_1(\la),C_2(\la)>0$ such
that, for large $n$,
$$
C_1(\la)\|u_n\|_2^{\frac{2}{r}}\le\|u_n\|_r\le
C_2(\la)\|u_n\|_2^{\frac{2}{r}}.
$$
Thus $\|u_n\|_p^p\leq C_3(\la)\|u_n\|_2^2$ and, by (\ref{s02})
again,
\begin{align*}
\|u_n\|_E^2+\int_{\R^3}\phi_{u_n}u_n^2+\la\|u_n\|_r^r\le
C_4(\la)\|u_n\|_2^2.
\end{align*}
Let $w_n=\frac{u_n}{\|u_n\|_E}$. The last inequality implies that
\begin{align}\label{p01}
\|w_n\|_2^2\geq(C_4(\la))^{-1}
\end{align}
and
\begin{align}\label{p02}
\int_{\R^3}\phi_{w_n}w_n^2\le C_5(\la)\|u_n\|_E^{-2}.
\end{align}
From \eqref{p02}, we have $\int_{\R^3}\phi_{w_n}w_n^2\rg 0$ as
$n\rg\iy$. Since $\|w_n\|_E=1$, we assume that $w_n\rg w$ weakly in
$E$ and strongly both in $L^\frac{12}{5}(\R^3)$ and in $L^2(\R^3)$.
Note that
\begin{align*}
\left|\int_{\R^3}(\phi_{w_n}w_n^2-\phi_ww^2)\right|
\leq&\int_{\R^3}|\phi_{w_n}-\phi_w|w_n^2+\int_{\R^3}\phi_w|w_n^2-w^2|\\
\leq&\|\phi_{w_n}-\phi_w\|_6\|w_n\|_{\frac{12}{5}}^2
+\|\phi_w\|_6\|w_n-w\|_{\frac{12}{5}}\|w_n+w\|_{\frac{12}{5}}.
\end{align*}
Since $w_n\to w$ strongly in $L^\frac{12}{5}(\R^3)$ and, by Lemma
\ref{l0}, $\phi_{w_n}\to\phi_w$ strongly in $L^6(\R^3)$, we have
$$
\int_{\R^3}\phi_ww^2=\lim_{n\to\infty}\int_{\R^3}\phi_{w_n}w_n^2=0,
$$
which implies $w=0$. But \eqref{p01} implies
$\|w\|_2^2\geq(C_4(\la))^{-1}$, and thus we have a contradiction and
finish the proof of the claim. The claim combined with Lemma
\ref{l3.22} implies $\|I_\la'(u_n)\|\rg 0$ as $n\rg\iy$, which is
again a contradiction.
\end{proof}

\begin{lemma}\lab{l4.3}
There exists $\e_1>0$ independent of $\la$ such that for
$\e\in(0,\e_1)$,
\begin{itemize}
\item [{\rm (1)}] $A_\la(\pa P_\e^-)\subset P_\e^-$
and every nontrivial solution $u\in P_\e^-$ is negative.
\item [{\rm (2)}] $A_\la(\pa P_\e^+)\subset P_\e^+$
and every nontrivial solution $u\in P_\e^+$ is positive.
\end{itemize}
\end{lemma}

\begin{lemma}\lab{l4.4}
There exists a locally Lipschitz continuous map $B_\la: E\setminus
K_\la\to E$, where $K_\la:={\rm Fix}(A_\la)$, such that
\begin{itemize}
\item [{\rm(1)}] $B_\la(\pa P_\e^+)\subset P_\e^+$,
$B_\la(\pa P_\e^-)\subset P_\e^-$ for $\e\in(0,\e_1)$;
\item [{\rm(2)}]
$\frac{1}{2}\|u-B_\la(u)\|_E\le\|u-A_\la(u)\|_E\le2\|u-B_\la(u)\|_E$
for
all $u\in E\setminus K_\la$;
\item [{\rm(3)}] $\langle I_\la'(u),u-B_\la(u)\rangle\ge
\frac{1}{2}\|u-A_\la(u)\|_E^2$ for all $u\in E\setminus K_\la$;
\item [{\rm(4)}] if $f$ is odd then $B_\la$ is odd.
\end{itemize}
\end{lemma}

We are ready to prove Theorem \ref{Th2}.

\begin{proof}[{\bf Proof of Theorem \ref{Th2} (Existence part)}]
{\bf Step 1.} We use Theorem A for $J=I_\la$. We claim that
$\{P_\e^+,P_\e^-\}$ is an admissible family of invariant sets for
the functional $I_\la$ at any level $c$. In view of the approach in
Section 3 and the fact that we have already had Lemmas 4.1-4.4, we
need only to prove that for any fixed $\la\in(0,1)$, $I_\la$
satisfies the (PS)-condition. Assume that there exist
$\{u_n\}\subset E$ and $c\in\R$ such that $I_\la(u_n)\rg c$ and
$I_\la'(u_n)\rg0$ as $n\rg\iy$. Similar to the proof of Lemma
\ref{l3.33}, we have, for $\gamma\in(4,q)$,
\begin{align*}
I_\la(u_n)&-\frac{1}{\gamma}\langle I_\la'(u_n),u_n\rangle\\
=&\left(\frac{1}{2}-\frac{1}{\gamma}\right)\|u_n\|_E^2
+\left(\frac{1}{4}-\frac{1}{\gamma}\right)\int_{\R^3}\phi_{u_n}u_n^2\\
&+\int_{\R^3}\big(\frac{1}{\gamma}f(u_n)u_n-F(u_n)\big)
+\la\left(\frac{1}{\gamma}-\frac{1}{r}\right)\|u_n\|_r^r.
\end{align*}
By $(f_1)$-$(f_2)$,
\begin{align*}&\|u_n\|_E^2+\int_{\R^3}\phi_{u_n}u_n^2+\la\|u_n\|_r^r
\le
C_1\left(|I_\la(u_n)|+\|u_n\|_E\|I_\la'(u_n)\|+\|u_n\|_p^p\right).
\end{align*}
Hence, for large $n$,
\begin{align*}
\lab{s02}\|u_n\|_E^2+\int_{\R^3}\phi_{u_n}u_n^2+\la\|u_n\|_r^r\le
C_2(1+\|u_n\|_p^p).
\end{align*}
As in the proof of Lemma \ref{l3.33}, one sees that $\{u_n\}$ is
bounded in $E$. Then, by Remark \ref{r1}, one can show that
$\{u_n\}$ has a convergent subsequence, verifying the
(PS)-condition.

{\bf Step 2.} Choose $v_1,v_2\in C_0^\infty(B_1(0))\setminus\{0\}$
such that $\mbox{supp}(v_1)\cap\mbox{supp}(v_2)=\emptyset$ and
$v_1\le 0,v_2\ge 0$, where $B_r(0):=\{x\in\R^3:\ |x|<r\}$. For
$(t,s)\in\Delta$, let
\begin{equation}\lab{ph1}
\vp_0(t,s)(\cdot):=R^2\left(tv_1(R\cdot)+sv_2(R\cdot)\right),
\end{equation}
where $R$ is a positive constant to be determined later. Obviously,
for $t,s\in[0,1]$, $\vp_0(0,s)(\cdot)=R^2sv_2(R\cdot)\in P_\e^+$ and
$\vp_0(t,0)(\cdot)=R^2tv_1(R\cdot)\in P_\e^-$. Similar to Lemma
\ref{l5.3}, for small $\e>0$,
$$
I_\la(u)\ge I_1(u)\ge\frac{\e^2}{2}\ \mbox{for}\ u\in\Sigma:=\pa
P_\e^+\cap\pa P_\e^-,\ \la\in(0,1),
$$
which implies that
$c_\la^\ast:=\inf_{u\in\Sigma}I_\la(u)\ge\frac{\e^2}{2}$ for
$\la\in(0,1)$. Let $u_t=\vp_0(t,1-t)$ for $t\in[0,1]$. Then a direct
computation shows that
\begin{itemize}
\item [(i)] $\int_{\R^3}|\na u_t|^2=R^3\int_{\R^3}\big(t^2|\na
v_1|^2+(1-t)^2|\na v_2|^2$\big),
\item [(ii)] $\int_{\R^3}|u_t|^2=R\int_{\R^3}\big(t^2v_1^2+(1-t)^2v_2^2$\big),
\item [(iii)] $\int_{\R^3}|u_t|^\mu
=R^{2\mu-3}\int_{\R^3}\big(t^\mu |v_1|^\mu+(1-t)^\mu|v_2|^\mu$\big),
\item [(iv)]
$\int_{\R^3}\phi_{u_t}|u_t|^2=R^3\int_{\R^3}\phi_{\ti{u}_t}|\ti{u}_t|^2$,
where $\ti{u}_t=tv_1+(1-t)v_2$.
\end{itemize}
Since $F(t)\ge C_3|t|^\mu-C_4$ for any $t\in\R$, by (i)-(iv) we
have, for $\la\in(0,1)$ and $t\in[0,1]$,
\begin{align*}
I_\la(u_t)\le&\frac{1}{2}\|u_t\|_E^2
+\frac{1}{4}\int_{\R^3}\phi_{u_t}|u_t|^2
-\int_{B_{R^{-1}}(0)}F(u_t)\\
\le&\frac{R^3}{2}\int_{\R^3}\big(t^2|\na v_1|^2+(1-t)^2|\na
v_2|^2\big)+\frac{R^3}{4}\int_{\R^3}\phi_{\ti{u}_t}|\ti{u}_t|^2\\
&+\frac{R}{2}\max_{|x|\le 1}V(x)\int_{\R^3}\big(t^2v_1^2+(1-t)^2v_2^2\big)\\
&-C_3R^{2\mu-3}\int_{\R^3}\big(t^\mu |v_1|^\mu+(1-t)^\mu
|v_2|^\mu\big)+C_5R^{-3}.
\end{align*}
Since $\mu>3$, one sees that $I_\la(u_t)\rg-\iy$ as $R\rg\iy$
uniformly for $\la\in(0,1),\ t\in[0,1]$. Hence, choosing $R$
independent of $\la$ and large enough, we have
$$
\sup\limits_{u\in\vp_0(\pa_0\Delta)}I_\la(u)<c_\la^\ast:=\inf_{u\in\Sigma}I_\la(u),\
\la\in(0,1).
$$
Since $\|u_t\|_2\rg\iy$ as $R\rg\iy$ uniformly for $t\in[0,1]$, it
follows from Lemma \ref{l5.2} that $\vp_0(\pa_0\Delta)\cap
M=\emptyset$ for $R$ large enough. Thus $\varphi_0$ with a large $R$
independent of $\lambda$ satisfies the assumptions of Theorem A.
Therefore, the number
$$
c_\la=\inf\limits_{\varphi\in\G}\sup\limits_{u\in\varphi(\triangle)\setminus
W}I_\la(u),
$$
is a critical value of $I_\la$ satisfying $c_\la\ge c_\la^\ast$ and
there exists $u_\la\in E\setminus (P_\e^+\cup P_\e^-)$ such that
$I_\la(u_\la)=c_\la$ and $I_\la'(u_\la)=0$.

{\bf Step 3.} Passing to the limit as $\lambda \to 0$. By the
definition of $c_\la$, we see that for $\la\in(0,1)$,
$$c_\la\le c(R):=\sup\limits_{u\in\varphi_0(\triangle)}I(u)<\iy.$$
We claim that $\{u_\la\}_{\la\in(0,1)}$ is bounded in $E$. We first
have
\begin{equation}\lab{ff1}
c_\la=\frac{1}{2}\int_{\R^3}\big(|\na
u_\la|^2+V(x)u_\la^2\big)+\frac{1}{4}\int_{\R^3}\phi_{u_\la}u_\la^2
-\int_{\R^3}\big(F(u_\la)+\frac{\la}{r}|u_\la|^r\big)
\end{equation}
and
\begin{equation}\lab{ff2}
\int_{\R^3}\big(|\na u_\la|^2+V(x)u_\la^2+\phi_{u_\la}u_\la^2 -u_\la
f(u_\la)-\la|u_\la|^r\big)=0.
\end{equation}
Moreover, we have the Pohoz$\check{\rm a}$ev identity
\begin{align}\lab{ff3}
\frac{1}{2}\int_{\R^3}|\na
u_\la|^2&+\frac{3}{2}\int_{\R^3}V(x)u_\la^2+\frac{1}{2}\int_{\R^3}u_\la^2\na
V(x)\cdot x\nonumber\\
&+\frac{5}{4}\int_{\R^3}\phi_{u_\la}u_\la^2
-\int_{\R^3}\big(3F(u_\la)+\frac{3\la}{r}|u_\la|^r\big)=0.
\end{align}
Multiplying (\ref{ff1}) by $3-\frac{\mu}{2}$, (\ref{ff2}) by $-1$
and (\ref{ff3}) by $\frac{\mu}{2}-1$ and adding them up, we obtain
\begin{align}\lab{ff4}
(3-\frac{\mu}{2})c_\la&=(\frac{\mu}{4}-\frac{1}{2})\int_{\R^3}\big(2V(x)+\na
V(x)\cdot x\big)u_\la^2\nonumber\\
&\ \ \ +(\frac{\mu}{2}-\frac{3}{2})\int_{\R^3}\phi_{u_\la}u_\la^2
+(1-\frac{\mu}{r})\la\int_{\R^3}|u_\la|^r\nonumber\\
&\ \ \ +\int_{\R^3}\big(u_\la f(u_\la)-\mu F(u_\la)\big).
\end{align}
Using $(V_1)$, $(f_3)$ and the fact that $3<\mu\leq p<r$, one sees
that $\{\int_{\R^3}\phi_{u_\la}u_\la^2\}_{\la\in(0,1)}$ is bounded.
From this fact it can be deduced from $(f_3)$, (\ref{ff1}), and
(\ref{ff2}) that $\{u_\la\}_{\la\in(0,1)}$ is bounded in $E$.

Assume that up to a subsequence, $u_\la\rg u$ weakly in $E$ as
$\la\rg0^+$. By Remark \ref{r1}, $u_\la\rg u$ strongly in
$L^q(\R^3)$ for $q\in[2,6)$. Then, by Lemma \ref{l0},
$\phi_{u_\la}\rg \phi_u$ strongly in $\mathcal{D}^{1,2}(\R^3)$. By a
standard argument, we see that $I'(u)=0$ and $u_\la\rg u$ strongly
in $E$ as $\la\rg 0^+$. Moreover, the fact that $u_\la\in E\setminus
(P_\e^+\cup P_\e^-)$ and $c_\la\ge\frac{\e^2}{2}$ for $\la\in(0,1)$
implies $u\in E\setminus (P_\e^+\cup P_\e^-)$ and
$I(u)\ge\frac{\e^2}{2}$. Therefore, $u$ is a sign-changing solution
of (\ref{q2}).
\end{proof}

In the following, we prove the existence of infinitely many
sign-changing solutions to (\ref{q2}). We assume that $f$ is odd.
Thanks to Lemmas 4.1-4.4, we have seen that $P_\e^+$ is a
$G-$admissible invariant set for the functional $I_\la\ (0<\la<1)$
at any level $c$.

\begin{proof}[{\bf Proof Theorem
\ref{Th2} (Multiplicity part)}] {\bf Step 1.} We construct $\vp_n$
satisfying the assumptions in Theorem B. For any $n\in\mathbb{N}$,
we choose $\{v_i\}_1^n\subset C_0^\infty(\R^3)\setminus\{0\}$ such
that
$\mbox{supp}(v_i)\cap\mbox{supp}(v_j)=\emptyset$ for $i\not=j$.
Define $\varphi_n\in C(B_n,E)$ as
\begin{equation}\lab{ph2}
\varphi_n(t)(\cdot)=R_n^2\sum_{i=1}^nt_iv_i(R_n\cdot),
\quad t=(t_1,t_2,\cdots,t_n)\in B_n,
\end{equation}
where $R_n>0$ is a large number independent of $\la$ such that
$\varphi_n(\pa B_n)\cap (P_\e^+\cap P_\e^-)=\emptyset$ and
$$
\sup\limits_{u\in\varphi_n(\pa
B_n)}I_{\la}(u)<0<\inf\limits_{u\in\Sg}I_\la(u).
$$
Obviously, $\varphi_n(0)=0\in P_\e^+\cap P_\e^-$ and
$\varphi_n(-t)=-\varphi_n(t)$ for $t\in B_n$.

{\bf Step 2.} For any $j\in\mathbb{N}$ and $\la\in(0,1)$, we define
$$
c_j(\la)=\inf\limits_{B\in\G_j}\sup\limits_{u\in B\setminus
W}I_\la(u),
$$
where $W:=P_\e^+\cup P_\e^-$ and $\G_j$ is as in Theorem B.
According to Theorem B, for any $0<\la<1$ and $j\ge 2$,
$$
0<\inf_{u\in\Sigma}I_\la(u):=c^\ast(\la)\le c_j(\la)\rg\iy\ \mbox{as}\ j\rg\iy
$$
and there exists $\{u_{\la,j}\}_{j\ge 2}\subset E\setminus W$ such
that $I_\la(u_{\la,j})=c_j(\la)$ and $I_\la'(u_{\la,j})=0$.

{\bf Step 3.} In a similar way to the above, for any fixed $j\ge 2$,
$\{u_{\la,j}\}_{\la\in(0,1)}$ is bounded in $E$. Without loss of
generality, we assume that $u_{\la,j}\rg u_j$ weakly in $E$ as
$\la\rg 0^+$. Observe that $c_j(\la)$ is decreasing in $\la$. Let
$c_j=\lim_{\la\rg 0^+}c_j(\la)$. Clearly $c_j(\la)\leq c_j<\infty$
for $\la\in(0,1)$. Then we may assume that $u_{\la,j}\rg u_j$
strongly in $E$ as $\la\rg0^+$ for some $u_j\in E\setminus W$ such
that $I'(u_j)=0$, $I(u_j)=c_j$. Since $c_j\ge c_j(\la)$ and
$\lim_{j\to\infty}c_j(\la)=\iy$, $\lim_{j\to\infty}c_j=\iy$.
Therefore, equation \eqref{q2} and thus system (\ref{q1}) has
infinitely many sign-changing solutions. The proof is completed.
\end{proof}

\vskip0.1in

\noindent{\bf Acknowledgements:}
J. Zhang thanks Dr. Zhenping Wang and Prof. Huansong Zhou for
letting him know their work \cite{Zhou}.

\end{document}